\documentclass[a4paper,11pt]{article}

\usepackage{amsmath,amsthm,amssymb,mathrsfs,latexsym,amsfonts}
\usepackage{graphicx,psfrag,epsfig}
\usepackage{bbm}
\usepackage[english]{babel}
\usepackage[latin1]{inputenc}
\usepackage{a4wide}

\newtheorem{theorem}{Theorem}[section]
\newtheorem{lemma}[theorem]{Lemma}
\newtheorem{proposition}[theorem]{Proposition}

\newcounter{paraga}[subsection]
\renewcommand{\theparaga}{{\bf\arabic{paraga}.}}
\newcommand{\paraga}{\medskip \addtocounter{paraga}{1} 
\noindent{\theparaga\ } }

\begin{document}

\def\MP{\,{<\hspace{-.5em}\cdot}\,}
\def\SP{\,{>\hspace{-.3em}\cdot}\,}
\def\PM{\,{\cdot\hspace{-.3em}<}\,}
\def\PS{\,{\cdot\hspace{-.3em}>}\,}
\def\EP{\,{=\hspace{-.2em}\cdot}\,}
\def\PP{\,{+\hspace{-.1em}\cdot}\,}
\def\PE{\,{\cdot\hspace{-.2em}=}\,}
\def\N{\mathbb N}
\def\C{\mathbb C}
\def\Q{\mathbb Q}
\def\R{\mathbb R}
\def\T{\mathbb T}
\def\A{\mathbb A}
\def\Z{\mathbb Z}
\def\demi{\frac{1}{2}}

\begin{titlepage}
\author{Abed Bounemoura~\footnote{abedbou@gmail.com, IHES, 35 route de Chartres, 91440 Bures-sur-Yvette (abed@ihes.fr)}}
\title{\LARGE{\textbf{The KAM theorem through Dirichlet's box and Khintchine's transference principles}}}
\end{titlepage}

\maketitle

\begin{abstract}
In this paper, we give a new proof of the classical KAM theorem which avoids small divisors and relies on two basic principles of Diophantine approximation: Dirichlet's box and Khintchine transference principles.  
\end{abstract}
 
\section{Introduction and statement of the result}\label{s1}

\subsection{Introduction}\label{s11}

\paraga KAM theory is part of the theory of perturbation of quasi-periodic motions in dynamical systems, which was initiated by Kolmogorov (\cite{Kol54}), Arnold (\cite{Arn63a},\cite{Arn63b}) and Moser (\cite{Mos62}) in the context of Hamiltonian systems. It is often considered ``as one of the most important achievements in the qualitative theory of ordinary differential equations in the whole second half of the twentieth century", and has been the subject of extensive investigations (see \cite{Sev03} and references therein).

\paraga A model problem, already considered by the founders of the theory (\cite{Kol53}, \cite{Arn61} and \cite{Mos66}), deals with vector fields on the torus and can be explained as follows. Let $\alpha$ be a non-zero vector in $\R^n$, and $X_\alpha=\alpha$ the constant vector field on the $n$-dimensional torus $\T^n=\R^n/\Z^n$ equals to $\alpha$. The integral curves of $X_\alpha$ describe, by definition, quasi-periodic motions with frequency $\alpha$, and the problem is to understand the fate of these solutions when a small but arbitrary perturbation $P$ is added to $X_\alpha$. In general, the vector field $X_\alpha$ is not structurally stable, as the perturbation will induce a shift of frequency so that the vector field $X_\alpha+P$ cannot be conjugated to $X_\alpha$. However, if the frequency $\alpha$ satisfies a Diophantine condition and $P$ is sufficiently regular, and if moreover we are allowed to first shift the frequency according to the perturbation, then the resulting vector field can indeed be conjugated to $X_\alpha$: this is the content of the classical KAM theorem, first proved by Arnold in \cite{Arn61}) (we refer to Theorem~\ref{th}, \S\ref{s12}, for a precise statement).

\paraga In the classical approach to the KAM theorem, and more generally in the classical theory of perturbation of quasi-periodic motions, a central role is played by the equation
\begin{equation}\label{eqlin}
[V,X_\alpha]=P-[P]_{\alpha}, \quad [P]_{\alpha}=\lim_{t \rightarrow +\infty}\frac{1}{t}\int_{0}^{t}(X_{\alpha}^{s})^* P ds
\end{equation}
where $V$ is the unknown vector field and $[P]_\alpha$ is the time-average of $P$ along the flow $X_\alpha^s$ of $X_\alpha$ ($(X_{\alpha}^{s})^* P$ is the pull-back of $P$ by $X_{\alpha}^{s}$). Assuming that $\alpha$ is non-resonant, that is if the components of $\alpha$ are independent over the field of rational numbers $\Q$, it follows from Birkhoff's ergodic theorem that $[P]_\alpha$ coincides with the space average, that is
\[ [P]_\alpha=[P]=\int_{\T^n}P \]
where the integration on $\T^n$ is with respect to its Haar measure. Equation~\eqref{eqlin}, which is usually called the homological equation, can be seen as the linearized version of the conjugacy equation we need to solve, and the time-one map $\Phi=V^1$ of the vector field $V$ is then used as a building block to construct the sought conjugacy by an iterative scheme. It is precisely in trying to solve the equation~\eqref{eqlin} that small divisors arise. Geometrically, one needs to integrate along the integral curves of $X_\alpha$, and these curves are not closed (they densely fill the torus $\T^n$). Analytically, one needs to invert the operator $[\,.\,,X_\alpha]$ acting on the space of smooth vector fields, and this operator is unbounded. Indeed, this operator can be diagonalized in a Fourier basis: letting $e_k(\theta)=e^{ik.\theta}$ for $k\in\Z^n$ and $V=(V_1,\dots,V_n)$, $P=(P_1,\dots,P_n)$, if $V_j=\sum_{k\in\Z^n}V_{j,k} e_k$ and $P_j=\sum_{k\in\Z^n}P_{j,k} e_k$ for $1\leq j \leq n$, then the solution is given by
\[ V_{j,0}=0, \quad V_{j,k}=(ik.\alpha)^{-1}P_{j,k}, \quad k\in \Z^n\setminus\{0\},  \]
and so the absolute values $|k.\alpha|^{-1}$ of its eigenvalues can be arbitrarily large if the supremum norm $|k|$ is arbitrarily large. It is then useful to quantify how the quantities $|k.\alpha|^{-1}$ grow with $|k|$ by introducing the function
\begin{equation}\label{eqpsi}
\Psi_\alpha(Q)=\max\left\{|k.\alpha|^{-1} \; | \; k\in \Z^n, \; 0<|k|\leq Q\right\}, \quad Q\geq 1. 
\end{equation} 
A vector $\alpha$ is said to satisfy a Diophantine condition if the function $\Psi_{\alpha}(Q)$ grows at most as a power of $Q$ (see \S\ref{s12} for a more precise definition). To overcome the effect of small divisors, for a given parameter $Q \geq 1$ a classical approach is to approximate $P$ by a trigonometric vector field $P_Q$ and to solve the approximate equation 
\begin{equation}\label{eqapp}
[V,X_\alpha]=P_Q-[P_Q]=P_Q-[P]
\end{equation}
using Fourier analysis as above: the norm of the vector field $V$, and hence the transformation $\Phi=V^1$, can be essentially controlled in terms of $\Psi(Q)$, and the term $P-P_Q$ is simply considered as an ``error" that causes no trouble in the iteration scheme. We refer to the very nice survey \cite{Pos01} for a detailed exposition of the classical approach in the Hamiltonian setting, and to \cite{Pos11} for a variant of the classical approach (proposed by Rüssmann in \cite{Rus10}, using a somehow ``optimal" approximation $P_Q$) in our setting. 

\paraga Recently, in a joint work with Fischler (\cite{BF12}), we proposed a fundamentally different approach to the classical theory of perturbation of quasi-periodic motions, replacing the small divisors problem by a method of periodic approximations. To explain this method, let us assume without loss of generality that $\alpha=(1,\alpha_1,\dots,\alpha_{n-1})$ with $|\alpha_j|\leq 1$ for $1\leq j \leq n$ (this can always be achieved, dividing $\alpha$ by its supremum norm $|\alpha|$ and re-ordering its components if necessary); such a simplification will enable us to deal with rational approximations only. So let $\omega \in \Q^n\setminus\{0\}$ be a rational vector with minimal denominator $q$, and consider the equation
\begin{equation}\label{eqlin2}
[V,X_\omega]=P-[P]_{\omega}, \quad [P]_{\omega}=\lim_{t \rightarrow +\infty}\frac{1}{t}\int_{0}^{t}(X_{\omega}^{s})^* P ds.
\end{equation}
It is easy to see that the time-average $[P]_{\omega}$ has the following simple expression
\[ [P]_\omega=q^{-1}\int_{0}^{q}P\circ X_{\omega}^t dt=\int_{0}^{1}P\circ X_{q\omega}^t dt. \]
Unlike equation~\eqref{eqlin}, equation~\eqref{eqlin2} is easily solved without Fourier expansions by the following simple integral formula
\[ V=q\int_{0}^{1}(P-[P]_\omega)\circ X_{q\omega}^t tdt \]
and there is no small divisors: geometrically, the integral curves of $X_\omega$ are closed, and analytically, the inverse operator of $[\,.\,,X_\omega]$ is bounded (by $q$, with respect to any translation-invariant norm on the space of vector fields on the torus). This was first used by Lochak in connection with the Nekhoroshev theorem (\cite{Loc92}). In \cite{BF12}, we proved (as a particular case) the following result in Diophantine approximation: if $\alpha$ is non-resonant, then there exists $n$ rational vectors $\omega_1,\dots,\omega_n$, of denominators $q_1,\dots,q_n$, such that $q_1\omega_1,\dots,q_n\omega_n$ form a basis of $\Z^n$ and such that, up to constants depending only on $n$, the distance $|\alpha-\omega_i|$ is bounded by $(q_iQ)^{-1}$ and the denominators $q_i$ are bounded by $\Psi_\alpha(Q)$, where $\Psi_\alpha$ is the function defined in~\eqref{eqpsi}. This result allows us to reduce the study of the equation~\eqref{eqlin} (or the approximate equation~\eqref{eqapp}) to the simpler equation~\eqref{eqlin2}, avoiding small divisors and the use of Fourier expansions as a consequence. Indeed, let $P_0=P$ and defined inductively $P_j=[P_{j-1}]_{\omega_j}$ for $1 \leq j \leq n$, then we solve the $n$ equations
\begin{equation}\label{eqlin22}
[V_j,X_{\omega_j}]=P_{j-1}-[P_{j-1}]_{\omega_j}, \quad 1 \leq j \leq n,
\end{equation}
which are of the type~\eqref{eqlin2}, and whose solutions are ``controlled" in terms of $q_j$, hence in terms of $\Psi_\alpha(Q)$. Then $|\alpha-\omega_j|$ are simply considered as ``errors", $P_n=[\cdots[P]_{\omega_1}\cdots]_{\omega_n}=[P]$ as one can easily check using the fact that $q_1\omega_1, \dots, q_n\omega_n$ is a basis of the lattice $\Z^n$, and the composition of time-one maps $V_1^1 \circ \cdot \circ V_n^1$ will give a transformation $\Phi$ which has a similar effect as the transformation constructed by solving the equation~\eqref{eqapp}. 

\paraga There is a duality between the classical approach and the approach given in \cite{BF12}. The classical approach deals with equation~\eqref{eqlin} by replacing the perturbation by a simpler (trigonometric) perturbation keeping the integrable vector field fixed (this is equation~\eqref{eqapp}), whereas in \cite{BF12}, we deal with equation~\eqref{eqlin} by replacing the integrable vector field by simpler (periodic) vector fields, keeping the perturbation fixed (these are equations~\eqref{eqlin22}). In the words of Lochak (see \cite{Loc02}), the classical approach pertains to linear Diophantine approximation, whereas \cite{BF12} should pertain to simultaneous Diophantine approximation.  

\paraga The purpose of this article is to present yet another approach to the KAM theorem, which is in some sense intermediate between the usual approach (for instance \cite{Pos11}) and the approach given in \cite{BF12}, as it mixes both linear and simultaneous Diophantine approximation. The idea is quite simple. Instead of using $n$ rational approximations with suitable properties (as was done in \cite{BF12}), the existence of which is non-trivial, we will just use one rational approximation, the existence of which is trivial by Dirichlet's box principle: for any $Q\geq 1$, there exists a rational vector $\omega$ of denominator $q$ such that $|\alpha-\omega| \leq (qQ)^{-1}$ and $1 \leq q\leq Q^{n-1}$. Then we approximate the integrable vector field $X_\alpha$ by the periodic vector field $X_\omega$, we solve equation~\eqref{eqlin2} and we consider $|\alpha-\omega|$ as an ``error". But this is clearly not sufficient, as the equation involves the average $[P]_\omega$ instead of $[P]$, so we need furthermore to be able to consider $[P]_\omega-[P]$ as an ``error". It is easy to see that the Fourier expansion of $[P]_\omega-[P]$ contains only harmonics associated to integers $k$ such that $k.\omega=0$, hence using the decay of Fourier coefficients of regular vector fields, $[P]_\omega-[P]$ can be considered small enough provided that $|k|$ is large enough. This will be proved in Proposition~\ref{proputile} below, using Khintchine transference principle which makes a connection between linear and simultaneous Diophantine approximation. One can say that our approach here avoids small divisors, since we deal only with the equation~\eqref{eqlin2}, even though we make use of Fourier expansions to estimate the remainder.
  
\subsection{Statement of the result}\label{s12}

\paraga Before stating the result, let us describe precisely the setting.

Let $n \in \N$, $n\geq 2$. Without loss of generality, we may assume that 
\[ \alpha=(1,\tilde{\alpha})\in \R \times \R^{n-1}, \quad \tilde{\alpha}=(\alpha_1,\dots,\alpha_{n-1}) \in [-1,1]^{n-1}. \]
First we need to impose a Diophantine condition on $\alpha$, and for simplicity we will assume that $\alpha$ satisfies a classical Diophantine condition. For $\tau \geq 0$, we define
\[ \Omega^{n-1}(\tau)=\left\{x \in [-1,1]^{n-1} \; | \; \exists \gamma>0, \; \forall k\in \Z^{n-1}\setminus \{0\}, \; ||k.x||_{\Z} \geq \gamma |k|^{-(1+\tau)(n-1)}  \right\}, \]
where $||\,.\,||_{\Z}$ denotes the distance to the lattice $\Z$, that is $||y||_{\Z}=\min_{k \in \Z}|y-k|$. Hence $\tilde{\alpha} \in \Omega^{n-1}(\tau)$ if and only if there exists a constant $\gamma=\gamma(\tilde{\alpha})$ such that for all $k\in \Z^{n-1}\setminus \{0\}$, $||k.\tilde{\alpha}||_{\Z} \geq \gamma |k|^{-(1+\tau)(n-1)}$, and it is not a restriction to assume that $\gamma \leq 1$.

Then we need to impose a regularity condition on the perturbation, and we will assume that it is real analytic. Recall that $\T^n=\R^n/\Z^n$, and we define $\T_{\C}^{n}=\C^n/\Z^n$. For $z=(z_1,\dots,z_n)\in \C^n$, we define $|z|=\max_{1\leq j \leq n}|z_j|$, and given $s>0$, we define a complex neighbourhood of $\T^n$ in $\T_{\C}^{n}$ by 
\[ \T^n_s=\{\theta\in \T_{\C}^{n}\; | \; |(\mathrm{Im}(\theta_1),\dots,\mathrm{Im}(\theta_n))|< s\}. \]  
Without loss of generality, we may also assume that $s\leq 1$.

For a small parameter $\varepsilon \geq 0$, we will consider bounded real-analytic vector fields on $\T^n_{s}$, of the form
\begin{equation}\label{H}
X=X_\alpha+P, \quad X_\alpha=\alpha=(1,\tilde{\alpha}), \quad \tilde{\alpha}\in \Omega^{n-1}(\tau), \quad |P|_{s}=\sup_{z\in \T^n_s}|P(z)|\leq \varepsilon. \tag{$*$}
\end{equation}
By real-analytic, we mean that the vector field is analytic and is real valued for real arguments. For any such vector field $Y$, we shall denote by $Y^t$ its time-$t$ map for values of $t\in\C$ which makes sense, and given another vector field $Z$, we denote by $[Y,Z]$ their Lie bracket. Moreover, for a real-analytic embedding $\Phi : \T^n_{r} \rightarrow \T_{s}^{n}$, $r\leq s$, and for a real-analytic vector field $Y$ which is well-defined on the image of $\Phi$, we let $\Phi^*Y$ be the pull-back $Y$, which is well-defined on $\T^n_{r}$.

\paraga We can finally state the result, whose content is the classical KAM theorem for constant vector fields on the torus first proved by Arnold.

\begin{theorem}\label{th}
Let $X$ be as in~\eqref{H}. Then there exist constants $\varepsilon_*\leq 1$ and $C_1,C_2 \geq 1$, which depend only on $n$, $s$ and $\tilde{\alpha}$, such that if $\varepsilon \leq \varepsilon_*$, there exist a unique constant $\beta\in\C^n$ and a real-analytic embedding $\Phi : \T^n_{s/2} \rightarrow \T^n_{s}$ such that
\[ \Phi^*(X+X_\beta)=X_\alpha \]
with the estimates
\[ |\Phi-\mathrm{Id}|_{s/2}\leq C_1\varepsilon, \quad |\beta|\leq C_2\varepsilon.\]    
\end{theorem}

Let us now add some remarks on the above theorem.

Firstly, we only considered vectors $\alpha$ satisfying a classical Diophantine condition, but as in \cite{BF12}, the approach can be easily extended to a more general class of frequency vectors, even though we don't know yet what is the weakest condition on $\alpha$ we can reach through our method.

Secondly, also as in \cite{BF12}, we have decided to formulate our result in the setting of perturbation of constant vector fields on the torus, but using further elementary techniques, for instance as described in \cite{Pos01}, we could have formulated our result in the context of perturbation of integrable Hamiltonian systems without any difficulties, provided the  Hamiltonian is real-analytic and the integrable Hamiltonian is non-degenerate in the sense of Kolmogorov.

Finally, the constants $\varepsilon_*$, $C_1$, $C_2$ in the statement depend on $n$, $s$ and on $\tilde{\alpha}$ only through $\tau$, $\gamma$ and the constant $\bar{\gamma}$ that appears in Theorem~\ref{thK}. We could have easily provided explicit values for these constants, but for clarity we decided not to do so, and in the sequel, when convenient, we will sometimes use a $\cdot$ in replacement of any constant $C \geq 1$ depending only on $n,s,\tau,\gamma$ and $\bar{\gamma}$: that is, an expression of the form $u \MP v$  means that there exists a constant $C\geq 1$, that depends only on the above set of parameters, such that $u \leq Cv$. Similarly, we will use the notations $u \PM v$, $u \EP v$ and $u \PE v$.

\section{Proof of Theorem~\ref{th}}\label{s2}

The proof of Theorem~\ref{th} will be given in \S\ref{s23}, based on a quasi-periodic averaging result we will prove in \S\ref{s22}. Such a deduction is classical, but the novelty lies in the proof of the quasi-periodic averaging result, which will simply follow from a consequence of Dirichlet's box and Khintchine transference principles, as explained in \S\ref{s21}, and from simple analytical estimates that are exposed in the Appendix~\ref{app}.  

\subsection{Dirichlet's approximation and Khintchine transference theorems}\label{s21}

\paraga We start by recalling the classical theorem of Dirichlet on the approximation of an arbitrary non-zero vector by vector with rational components.

\begin{theorem}[Dirichlet]\label{thD}
Let $x\in \R^{n-1} \setminus \{0\}$, and $Q\geq 1$. Then there exists $(q,p)\in \N \times \Z^{n-1}$ such that
\[ |qx-p| \leq Q^{-1}, \quad 1 \leq q \leq Q^{n-1}. \]
\end{theorem} 

For a proof we refer to \cite{Sch} or \cite{Cas57}.

\paraga For $\tau \geq 0$, we already defined
\[ \Omega^{n-1}(\tau)=\left\{x \in [-1,1]^{n-1} \; | \; \exists \gamma>0, \; \forall k\in \Z^{n-1}\setminus \{0\}, \; ||k.x||_{\Z} \geq \gamma |k|^{-(1+\tau)(n-1)}  \right\}, \]
and now we define
\[ \Omega_{n-1}(\tau)=\left\{x \in [-1,1]^{n-1} \; | \; \exists \bar{\gamma}>0, \; \forall q\in \N^*, \; ||qx||_{\Z^{n-1}} \geq \bar{\gamma} q^{-(1+\tau)(n-1)^{-1}}  \right\} \]
where $||\,.\,||_{\Z^{n-1}}$ denotes the distance to the lattice $\Z^{n-1}$, that is $||y||_{\Z^{n-1}}=\min_{k \in \Z^{n-1}}|y-k|$. We shall use the following statement, which is a particular case of a transference theorem of Khintchine.

\begin{theorem}[Khintchine]\label{thK}
For $\tau\geq 0$, $\Omega^{n-1}(\tau) \subseteq \Omega_{n-1}((n-1)\tau)$. That is, if $\tilde{\alpha} \in \Omega^{n-1}(\tau)$, then there exists $\bar{\gamma}=\bar{\gamma}(\tilde{\alpha})$ such that for any $q\in \N^*$, 
\[ ||q\tilde{\alpha}||_{\Z^{n-1}} \geq \bar{\gamma} q^{-(1+(n-1)\tau)(n-1)^{-1}}.\]
\end{theorem}  

More generally, we have $\Omega_{n-1}\left(\tau((n-2)\tau+n-1)^{-1}\right) \subseteq \Omega^{n-1}(\tau) \subseteq \Omega_{n-1}((n-1)\tau)$ for any $\tau \geq 0$, which implies in particular that $\Omega^{n-1}(0)=\Omega_{n-1}(0)$. For a proof, we refer to \cite{Sch} or \cite{Cas57} (see also \cite{Ber03} for another related and interesting transference result).   

\paraga The proposition below, which is a consequence of Dirichlet's and Khintchine's Theorems, will be our main tool to derive our quasi-periodic averaging result.

\begin{proposition}\label{proputile}
Let $\alpha=(1,\tilde{\alpha})\in \R^n\setminus\{0\}$ with $\tilde{\alpha} \in \Omega^{n-1}(\tau)$, and $Q\geq1$. Then there exists a vector $\omega=(1,q^{-1}p) \in \Q^n$, with $(q,p)\in \N \times \Z^{n-1}$, such that:
\begin{itemize}
\item[$(i)$] $|\alpha-\omega|=|\tilde{\alpha}-q^{-1}p| \leq (qQ)^{-1}$, $1 \leq q < Q^{n-1}$;
\item[$(ii)$] for any $k\in\Z^n\setminus\{0\}$ satisfying $k.\omega=0$, we have $|k|\geq \gamma^* Q^{a^{-1}}$ with 
\[ \gamma^*=\left(n^{-1}\gamma\bar{\gamma}^{(n-1)(1+(n-1)\tau)^{-1}}\right)^{(n+(n-1)\tau)^{-1}}, \quad a=1+(n-1)\tau.\]
\end{itemize} 
\end{proposition}

\begin{proof}
The first part of the statement follows easily from Theorem~\ref{thD} applied to $x=\tilde{\alpha}$, so it remains to prove the second part. Since $\tilde{\alpha} \in \Omega^{n-1}(\tau)$, we can use Theorem~\ref{thK} and together with $(i)$, we obtain
\[ Q^{-1} \geq |q\tilde{\alpha}-p|\geq ||q\tilde{\alpha}||_{\Z^{n-1}} \geq \bar{\gamma} q^{-(1+(n-1)\tau)(n-1)^{-1}}  \]
which implies
\begin{equation}\label{estq}
q \geq (\bar{\gamma}Q)^{(n-1)(1+(n-1)\tau)^{-1}}.
\end{equation}
Now let $k\in\Z^n \setminus\{0\}$ such that $k.\omega=0$, and let us write $k=(k_0,\tilde{k}) \in \Z\times\Z^{n-1}$. Necessarily $\tilde{k}$ is non-zero. Now $k.\omega=0$ is equivalent to $qk.\omega=qk_0+\tilde{k}.p=0$, hence
\begin{equation}\label{ega}
q(k_0+\tilde{k}.\tilde{\alpha})=q(k_0+\tilde{k}.\tilde{\alpha})-(qk_0+\tilde{k}.p)=(q\tilde{\alpha}-p).\tilde{k}.
\end{equation}
On the one hand, using Cauchy-Schwarz inequality and $(i)$, we have
\[ |(q\tilde{\alpha}-p).\tilde{k}|\leq n |q\tilde{\alpha}-p||\tilde{k}|\leq nQ^{-1}|\tilde{k}|, \]
and on the other hand, using the fact that $\tilde{\alpha} \in \Omega^{n-1}(\tau)$, we have
\[ |q(k_0+\tilde{k}.\tilde{\alpha})|\geq q||\tilde{k}.\tilde{\alpha}|| \geq q\gamma|\tilde{k}|^{-(1+\tau)(n-1)}. \]
These last two inequalities, together with the equality~\eqref{ega}, implies
\[ |\tilde{k}|^{(1+\tau)(n-1)+1}=|\tilde{k}|^{n+(n-1)\tau} \geq n^{-1}\gamma qQ, \]
and using~\eqref{estq}, this gives
\begin{eqnarray*}
|\tilde{k}|^{n+(n-1)\tau} & \geq & n^{-1} \gamma \bar{\gamma}^{(n-1)(1+(n-1)\tau)^{-1}} Q^{(n-1)(1+(n-1)\tau)^{-1}+1} \\
& = & n^{-1} \gamma \bar{\gamma}^{(n-1)(1+(n-1)\tau)^{-1}} Q^{(n+(n-1)\tau)(1+(n-1)\tau)^{-1}}.
\end{eqnarray*}
This proves that $|\tilde{k}|\geq \gamma^* Q^{a^{-1}}$, with $\gamma^*$ and $a$ as in the statement, and this proves $(ii)$ since $|k|\geq |\tilde{k}|$.
\end{proof}

\subsection{Quasi-periodic averaging}\label{s22}

Recall that $a=1+(n-1)\tau \geq 1$ have been defined in the statement of Proposition~\ref{proputile}, now we define three additional constants
\begin{equation}\label{cons}
b=4^{na}> 1, \quad c=b^{-1}(1-b^{-1})^{-1} < 1, \quad d=c+1=(1-b^{-1})^{-1}\geq 1.
\end{equation}
Our aim here is to prove the following quasi-periodic averaging result.

\begin{proposition}\label{quasi}
Let $Y=X+S$, with $X$ as in~\eqref{H} and $|S|_s \leq d\varepsilon$. For any $0 < \sigma < s$ and $Q \geq 1$, assume that
\begin{equation}\label{thre}
Q^n \varepsilon \leq 1, \quad Q^{-1}\sigma^{-1} \PM 1, \quad Q\sigma^{1-n}e^{-\gamma^*Q^{a^{-1}}\sigma/2} \leq 1.
\end{equation}
Then there exists a real analytic embedding $\Phi_1 : \T^n_{s-\sigma} \rightarrow \T^n_{s}$ such that
\[ \Phi_1^* Y=X_\alpha+S+[P]+P^+, \quad  [P]=\int_{\T^n}P, \]
with the estimates 
\[ |\Phi_1-\mathrm{Id}|_{s-\sigma}\leq Q^{n-1}\varepsilon, \quad |P^+|_{s-\sigma}\leq b^{-1}\varepsilon. \]
\end{proposition}

The above statement does not concern $X$ but the modified vector field $X+S$, where $S$ is a sufficiently small arbitrary vector field. $S$ does not play any role here, but it will become useful for the proof of Theorem~\ref{th} to deal with the shift of frequency (caused by the average $[P]$) at each step of the iterative scheme. The choice of the constant $b>1$ (and subsequently the choice of $c$ and $d$) is rather arbitrary, but was made in order to simplify the proof of Theorem~\ref{th}.  

\begin{proof}[Proof of Proposition~\ref{quasi}]
Since $\alpha=(1,\tilde{\alpha})\in \R^n\setminus\{0\}$ with $\tilde{\alpha} \in \Omega^{n-1}(\tau)$, for $Q\geq1$ we can apply Proposition~\ref{proputile}: there exists a vector $\omega=(1,q^{-1}p) \in \Q^n$, with $(q,p)\in \N \times \Z^{n-1}$, and by~$(i)$,
\[ |\alpha-\omega|=|\tilde{\alpha}-q^{-1}p| \leq (qQ)^{-1}, \quad 1 \leq q \leq Q^{n-1}. \]
We set
\[ [P]_\omega=\int_{0}^{1}P\circ X_{q\omega}^t dt, \quad V=q\int_{0}^{1}(P-[P]_\omega)\circ X_{q\omega}^t tdt. \]
Since $|P|_s\leq \varepsilon$, then obviously $|[P]_\omega|_s\leq \varepsilon$ and $|V|_s\leq q\varepsilon$. Since $q\leq Q^{n-1}$, from the first and second part of~\eqref{thre}, we have in particular 
\begin{equation}\label{dist}
q\varepsilon \PM Q^{n-1}\varepsilon \PM Q^{-1} \PM \sigma
\end{equation}
hence by Lemma~\ref{tech1} (applied with $\varsigma=\sigma$), the map $V^1 : \T^n_{s-\sigma} \rightarrow \T^n_{s}$ is a well-defined real-analytic embedding and  
\[ |V^1-\mathrm{Id}|_{s-\sigma}\leq |V|_s\leq q\varepsilon \leq Q^{n-1}\varepsilon.\] 
Let us define $\varpi=\alpha-\omega$ so that $Y=X+S=X_\alpha+S+P=X_\omega+X_\varpi+S+P$. Now we can write
\begin{equation}\label{tay}
(V^1)^*Y=(V^1)^*X_\omega+(V^1)^*(X_\varpi+S+P),
\end{equation}
and using the general equality
\[ \frac{d}{dt}(V^t)^*F=(V^t)^*[F,V] \]
for an arbitrary vector field $F$, we can apply Taylor's formula with integral remainder to the right-hand side of~\eqref{tay}, at order two for the first term and at order one for the second term, and we get
\[ (V^1)^*Y=X_\omega+[X_\omega,V]+\int_{0}^{1}(1-t)(V^t)^*[[X_\omega,V],V]dt+X_\varpi+S+P+\int_{0}^{1}(V^t)^*[X_\varpi+S+P,V]dt. \]
Now let us check that the equality $[V,X_\omega]=P-[P]_\omega$ holds true:  let us denote $G=P-[P]_\omega$ and $DV$ the differential of $V$, then since $X_\omega$ is a constant vector field, we have
\[ [V,X_\omega]=DV.\omega=q\int_{0}^{1}D(G\circ X_{q\omega}^{t}).\omega tdt=\int_{0}^{1}D(G\circ X_{q\omega}^{t}).q\omega tdt \]
so using the chain rule
\[ [V,X_\omega]=\int_{0}^{1}\frac{d}{dt}(G\circ X_{q\omega}^{t}) tdt \]
and an integration by parts
\[ [V,X_\omega]=\left.(G\circ X_{q\omega}^{t})t\right\vert_{0}^{1}-\int_{0}^{1}G\circ X_{q\omega}^{t}dt=G,  \]
where in the last equality, $G\circ X_{q\omega}^{1}=G$ since $q\omega\in\Z^n$ and the integral vanishes since $[G]=0$. So using the equality $[V,X_\omega]=P-[P]_\omega$, that can be written as $[X_\omega,V]+P=[P]_\omega$, we have 
\[ (V^1)^*Y=X_\omega+X_\varpi+S+[P]_\omega+\int_{0}^{1}(1-t)(V^t)^*[[X_\omega,V],V]dt+\int_{0}^{1}(V^t)^*[X_\varpi+S+P,V]dt,\]
and if we set
\[ P_t=tP+(1-t)[P]_\omega, \quad \tilde{P}=\int_{0}^{1}(V^t)^*[X_\varpi+S+P_t,V]dt, \quad P^+=\tilde{P}+[P]_\omega-[P] \]
and use again the equality $[X_\omega,V]=[P]_\omega-P$ we eventually obtain
\[ (V^1)^*Y=X_\omega+X_\varpi+S+[P]_\omega+\tilde{P}=X_\alpha+S+[P]+P^+.\]
Letting $\Phi_1=V^1$, it remains only to estimate $P^+$.

Let us first estimate $\tilde{P}$, and for that let us write $U=[X_\varpi+S+P_t,V]$. This is just a sum of three Lie brackets, and since
\[ |X_\varpi|_s=|\varpi|=|\alpha-\omega| \leq (qQ^{-1}), \quad |S|_s \leq d\varepsilon, \quad |P_t|_s\leq \varepsilon, \quad |V|_s \leq q\varepsilon, \] 
each of them can be estimated by Lemma~\ref{tech2} (applied with $\varsigma=\sigma/2$) and we obtain
\[ |U|_{s-\sigma/2} \MP \sigma^{-1} q\varepsilon ((qQ)^{-1}+d\varepsilon+\varepsilon) \MP \sigma^{-1}Q^{-1}\varepsilon + \sigma^{-1}q\varepsilon^2 \MP \sigma^{-1}Q^{-1}\varepsilon  \]
where the last inequality follows from~\eqref{dist}, as $q\varepsilon \leq Q^{n-1}\varepsilon \leq Q^{-1}$. Then using~\eqref{dist} again, we can apply Lemma~\ref{tech3} (with $\varsigma=\sigma/2$ and $s-\sigma/2$ instead of $s$) to obtain
\[ |\tilde{P}|_{s-\sigma} \leq \sup_{|t|\leq 1}|(V^t)^*U|_{s-\sigma} \MP |U|_{s-\sigma/2} \MP \sigma^{-1}Q^{-1}\varepsilon.  \]
Now let us estimate $[P]_\omega-[P]$, but first observe that if $P=(P_1,\dots,P_n)$ and $P_j=\sum_{k\in\Z^n}P_{j,k}e_k$ is the Fourier expansion of $P_j$ for $1\leq j \leq n$ (recall that $e_k(\theta)=e^{ik.\theta}$, for $k\in\Z^n$), then
\[ [P]=(P_{1,0},\dots,P_{n,0}) \]
and if $[P]_\omega=(P_1^\omega,\dots,P_n^\omega)$ and $P_j^\omega=\sum_{k\in\Z^n}P_{j,k}^\omega e_k$, then
\begin{equation}
P_{j,k}^\omega=
\begin{cases}
P_{j,k}, & k.\omega=0, \\
0, & k.\omega \neq 0.
\end{cases}
\end{equation}
The first assertion is obvious, and the second follows from the computation below, where we denotes $\bar{e}_{k}(\theta)=e^{-ik.\theta}$:
\begin{eqnarray*}
P_{j,k}^\omega & = & \int_{\T^n}\bar{e}_k P_j^\omega =\int_{\T^n}\bar{e}_k \int_{0}^{1} P_j \circ X^t_{q\omega} dt = \int_{\T^n}\bar{e}_k \int_{0}^{1} \sum_{l\in\Z^n} P_{j,l}e_le^{itl.q\omega} dt \\
& = & \sum_{l\in\Z^n} P_{j,l} \int_{\T^n} \bar{e}_k e_l \int_{0}^{1} e^{itl.q\omega} dt = P_{j,k}\int_{0}^{1} e^{itk.q\omega}. 
\end{eqnarray*}
Now by Proposition~\ref{proputile}, $(ii)$, if $k.\omega=0$ and $k \neq 0$, then $|k|\geq \gamma^* Q^{a^{-1}}$, so we can apply Lemma~\ref{tech4} (with $K=\gamma^* Q^{a^{-1}}$ and $\varsigma=\sigma$) to obtain the estimate
\[ |[P]_\omega-[P]| \MP \sigma^{-n} e^{-\gamma^* Q^{a^{-1}}\sigma/2}|P|_s \MP \sigma^{-n} e^{-\gamma^* Q^{a^{-1}}\sigma/2} \varepsilon  \]
and therefore, using the last part of~\eqref{thre}, 
\[ |[P]_\omega-[P]|_{s-\sigma} \MP Q^{-1}\sigma^{-1}\varepsilon. \]
Eventually, choosing the implicit constant in the second part of~\eqref{thre} sufficiently large, we obtain
\[ |P^+|_{s-\sigma}\leq |\tilde{P}|_{s-\sigma}+|[P]_\omega-[P]|_{s-\sigma} \MP Q^{-1}\sigma^{-1}\varepsilon \leq b^{-1}\varepsilon \]
which ends the proof.
\end{proof}

\subsection{Proof of Theorem~\ref{th}}\label{s23}

For a given $r\geq 0$, let
\[ B_r(\alpha)=\{ x \in \C^n \; | \; |\alpha-x|\leq r \}. \]
The proposition below is just a more convenient reformulation of Proposition~\ref{quasi}, where we recall that the constants $c$ and $d$ have been defined in~\eqref{cons}.

\begin{proposition}\label{step}
Let $X$ be as in~\eqref{H}. For any $0 < \sigma < s$ and $Q \geq 1$, assume that
\begin{equation}\label{thre2}
Q^n \varepsilon \leq 1, \quad Q^{-1}\sigma^{-1} \PM 1, \quad Q\sigma^{1-n}e^{-\gamma^*Q^{a^{-1}}\sigma/2} \leq 1.
\end{equation}
Then there exist an embedding $\varphi_1 : B_{c\varepsilon}(\alpha) \rightarrow B_{d\varepsilon}(\alpha)$ and a real analytic embedding $\Phi_1 : \T^n_{s-\sigma} \rightarrow \T^n_{s}$ such that for all $x \in B_{c\varepsilon}(\alpha)$,
\[ \Phi_1^* (X_{\varphi(x)}+P)=X_x+P^+ \]
with the estimates 
\[ |\Phi_1-\mathrm{Id}|_{s-\sigma}\leq Q^{n-1}\varepsilon, \quad |P^+|_{s-\sigma}\leq b^{-1}\varepsilon. \]
\end{proposition}

\begin{proof}
Recall that $[P]$ is a constant, $[P]\in\C^n$, and as $|P|_s\leq \varepsilon$, then $|[P]|\leq \varepsilon$. So let us define $\varphi_1 : B_{c\varepsilon}(\alpha)\rightarrow B_{d\varepsilon}(\alpha)$ to be the translation 
\[ \varphi_1(x)=x-[P], \quad x\in B_{\varepsilon}(\alpha).\]
Take any $x\in B_{c\varepsilon}(\alpha)$, then $|\varphi_1(x)-\alpha|\leq |x-\alpha|+|[P]| \leq (c+1)\varepsilon=d\varepsilon$ and we can write
\[ X_{\varphi_1(x)}+P=X_{\alpha}+X_{\varphi_1(x)-\alpha}+P \]
and by condition~\eqref{thre}, we can apply Proposition~\ref{quasi} with $S=X_{\varphi_1(x)-\alpha}$, $|S|_s\leq d\varepsilon$, to find a real analytic embedding $\Phi_1 : \T^n_{s-\sigma} \rightarrow \T^n_{s}$ such that 
\[ \Phi_1^*(X_{\varphi_1(x)}+P)=X_{\varphi_1(x)}+[P]+P^+=X_{\varphi_1(x)+[P]}+P^+=X_{x}+P^+ \]
with the estimates 
\[ |\Phi_1-\mathrm{Id}|_{s-\sigma}\leq Q^{n-1}\varepsilon, \quad |P^+|_{s-\sigma}\leq b^{-1}\varepsilon. \] 
This was the statement to prove.
\end{proof}

We will now use Proposition~\ref{step} as the building block of an iterative scheme leading to the proof of Theorem~\ref{th}.

\begin{proof}
For a constant $Q\geq 1$ to be chosen below, we define, for any $m\in\N$,
\[ \varepsilon_m=b^{-m} \varepsilon, \quad Q_m=b^{n^{-1}m}Q=4^{am}Q, \quad \sigma_m=2^{-m-2}s. \]
We also define $s_0=s$ and $s_{m+1}=s_m-\sigma_m$ for $m\in\N$, and for convenience, we let $\varepsilon_{-1}=d\varepsilon$. We claim that for $Q_0=Q$ sufficiently large, the conditions
\begin{equation}\label{thre3}
Q_m^n \varepsilon_m \leq 1, \quad Q_m^{-1}\sigma_m^{-1} \PM 1, \quad Q_m\sigma_m^{1-n}e^{-\gamma^*Q_m^{a^{-1}}\sigma_m/2} \leq 1
\end{equation}
are satisfied for any $m\in\N$. 

Indeed, we can choose $Q \EP 1$ with a sufficiently large implicit constant so that the last two conditions in~\eqref{thre3} are satisfied for $m=0$, then if we define $\varepsilon_*=Q^{-n}$, the threshold $\varepsilon \leq \varepsilon_*$ implies the first condition in~\eqref{thre3} for $m=0$. Next it is easy to see that $Q_m^n\varepsilon_m=Q_0^n\varepsilon_0$ and $Q_m\sigma_m=4^{am}2^{-m}Q_0\sigma_0\geq Q_0\sigma_0$ since $a\geq 1$, therefore the first two conditions in~\eqref{thre3} are satisfied for any $m\in\N$. For the last condition, let
\[ M_m=Q_m\sigma_m^{1-n}e^{-\gamma^*Q_m^{a^{-1}}\sigma_m/2}, \]
observing that $Q_m^{a^{-1}}\sigma_m=2^m Q_0^{a^{-1}}\sigma_0$, we have for any $m\geq 1$,
\[ M_{m+1}M_m^{-1}=4^a 2^{n-1}e^{-\gamma^*Q_m^{a^{-1}}\sigma_m/2} \leq 4^a 2^{n-1}e^{-\gamma^*Q_0^{a^{-1}}\sigma_0} \]
and hence $M_{m+1}M_m^{-1} \leq 1$, up to taking a larger implicit constant if necessary in the definition $Q \EP 1$. So this proves the claim.

Next we claim that for any $m\in\N$, there exist an embedding $\varphi^m : B_{c\varepsilon_{m-1}}(\alpha)\rightarrow B_{d\varepsilon}(\alpha)$ and a real analytic embedding $\Phi^m : \T^n_{s_{m}} \rightarrow \T^n_{s}$ such that for all $x_m\in B_{c\varepsilon_{m-1}}(\alpha)$,
\[ (\Phi^m)^* (X_{\varphi^m(x_m)}+P)=X_{x_m}+P_m \]
with the estimates 
\begin{equation}\label{estim}
|\Phi^m-\mathrm{Id}|_{s_m}\leq \sum_{j=0}^{m-1}Q_j^{n-1}\varepsilon_j=Q^{n-1}\varepsilon\sum_{j=0}^{m-1}b^{-jn^{-1}}, \quad |P_{m}|_{s_{m}}\leq \varepsilon_{m}. 
\end{equation}
Indeed, for $m=0$, choosing $\varphi^0$ and $\Phi^0$ to be the identity, and $P_0=P$, there is nothing to prove. If we assume that the statement holds true for some $m\in\N$, then by \eqref{thre3} we can apply Proposition~\ref{step} (with $\varepsilon_m,Q_m$ and $\sigma_m$ instead of $\varepsilon,Q$ and $\sigma$) to the resulting vector field and an embedding $\varphi_{m+1} : B_{c\varepsilon_{m}}(\alpha)\rightarrow B_{d\varepsilon_{m}}(\alpha)$ and a real analytic embedding $\Phi_{m+1} : \T^n_{s_{m+1}} \rightarrow \T^n_{s_m}$ are constructed. It is then sufficient to let $\varphi^{m+1}=\varphi^m\circ \varphi_{m+1}$, which is well defined since $d\varepsilon_{m+1}=c\varepsilon_m$, $\Phi^{m+1}=\Phi^m\circ \Phi_{m+1}$, $P_{m+1}=P_m^+$ and the estimates~\eqref{estim} are obvious.

To conclude, note that
\begin{equation*}
\lim_{m \rightarrow +\infty}\varepsilon_m=0, \quad \lim_{m \rightarrow +\infty}s_m=s/2, 
\end{equation*}
and together with the estimates~\eqref{estim}, when $m$ goes to infinity, $\varphi^m$ converges to a trivial map $\varphi : \{\alpha\} \rightarrow B_{d\epsilon}(\alpha)$, $P_m$ converges to zero uniformly on every compact subsets of $\T^n_{s/2}$, while $\Phi^m$ converges to an embedding $\Phi : \T^n_{s/2} \rightarrow \T^n_s$, uniformly on every compact subsets of $\T^n_{s/2}$. Since the space of real-analytic functions is closed for the topology of uniform convergence on compact subsets, $\Phi$ is real-analytic, with the estimate
\[ |\Phi-\mathrm{Id}|_{s/2}\leq \left(1-b^{-n^{-1}}\right)^{-1}Q^{n-1}\varepsilon.\]
Finally, at the limit we have $\Phi^*(X_{\varphi(\alpha)}+P)=X_\alpha$, hence if $\beta \in \C^n$ is the unique vector such that $\varphi(\alpha)=\alpha+\beta$, then $|\beta|\leq d\varepsilon$ and  
\[  \Phi^*(X+X_\beta)=X_\alpha. \]
This was the statement to prove. 
\end{proof}

\appendix

\section{Technical estimates}\label{app}

In this appendix, we recall some technical estimates that we used for the proof of Proposition~\ref{quasi}. These estimates are classical: proofs of Lemma~\ref{tech1}, Lemma~\ref{tech2} and Lemma~\ref{tech3} can be found in \cite{BF12} and a proof of Lemma~\ref{tech1} can be found in \cite{BGG85}.

\begin{lemma}\label{tech1}
Let $V$ be a bounded real-analytic vector field on $\T^n_s$, $0<\varsigma<s$ and $\tau=\varsigma |V|_{s}^{-1}$. For $t\in \C$ such that $|t|< \tau$, the map $V^t : \T^n_{s-\varsigma} \rightarrow \T^n_s$ is a well-defined real-analytic embedding, and we have 
\[ |V^t-\mathrm{Id}|_{s-\varsigma}\leq |V|_s, \quad |t|<\tau. \]
Moreover, $V^t$ depends analytically on $t$, for $|t|<\tau$.
\end{lemma} 

\begin{lemma}\label{tech2}
Let $X$ and $V$ be two bounded real-analytic vector fields on $\T^n_s$, and $0<\varsigma<s$. Then
\[ |[X,V]|_{s-\varsigma}\leq 2\varsigma^{-1}|X|_s|V|_s. \]
\end{lemma}

\begin{lemma}\label{tech3}
Let $X$ and $V$ be two bounded real-analytic vector fields on $\T^n_s$, and $0<\varsigma<s$. Assume that $|V|_s \leq (4e)^{-1}\varsigma$. Then for all $|t|\leq 1$, 
\[ |(V^t)^*X|_{s-\varsigma}\leq 2|X|_s. \]
\end{lemma}

\begin{lemma}\label{tech4}
Let $X$ be a bounded real-analytic vector field on $\T^n_s$, $0<\varsigma<s$ and $K \geq 1$. If $X=(X_1,\dots,X_n)$ and $X_j=\sum_{k \in \Z^n}X_{j,k} e_k$, then the vector field $X^K=(X_1^K,\dots,X_n^K)$ defined by 
\[ X_j^K=\sum_{k \in \Z^n, \; |k|\geq K} X_{j,k} e_k, \quad 1 \leq j \leq n\]
satisfies
\[ |X^K|_{s-\varsigma} \leq C_n \varsigma^{-n} e^{-K\varsigma/2}|X|_s \]
for a constant $C_n \geq 1$ which depends only on $n$.
\end{lemma}

\addcontentsline{toc}{section}{References}
\bibliographystyle{amsalpha}
\bibliography{KDK}

\providecommand{\bysame}{\leavevmode\hbox to3em{\hrulefill}\thinspace}
\providecommand{\MR}{\relax\ifhmode\unskip\space\fi MR }
\providecommand{\MRhref}[2]{%
  \href{http://www.ams.org/mathscinet-getitem?mr=#1}{#2}
}
\providecommand{\href}[2]{#2}
\begin{thebibliography}{BGG85}

\bibitem[Arn61]{Arn61}
V.I. Arnol'd, \emph{Small denominators. {I}. {M}apping the circle onto itself},
  Izv. Akad. Nauk SSSR Ser. Mat. \textbf{25} (1961), 21--86.

\bibitem[Arn63a]{Arn63a}
\bysame, \emph{{Proof of a theorem of {A}.{N}. Kolmogorov on the invariance of
  quasi-periodic motions under small perturbations}}, Russ. Math. Surv.
  \textbf{18} (1963), no.~5, 9--36.

\bibitem[Arn63b]{Arn63b}
\bysame, \emph{{Small denominators and problems of stability of motion in
  classical and celestial mechanics}}, Russ. Math. Surv. \textbf{18} (1963),
  no.~6, 85--191.

\bibitem[Ber03]{Ber03}
P.~Bernard, \emph{Une propri\'et\'e de transfert en approximation
  diophantienne}, Ann. Fac. Sci. Toulouse Math. (6) \textbf{12} (2003), no.~4,
  453--463.

\bibitem[BF12]{BF12}
A.~Bounemoura and S.~Fischler, \emph{A diophantine duality applied to the {KAM}
  and {Nekhoroshev} theorems}, Math. Z. (2012), to appear.

\bibitem[BGG85]{BGG85}
G.~Benettin, L.~Galgani, and A.~Giorgilli, \emph{{A proof of {N}ekhoroshev's
  theorem for the stability times in nearly integrable {H}amiltonian systems}},
  Celestial Mech. \textbf{37} (1985), 1--25.

\bibitem[Cas57]{Cas57}
J.W.S. Cassels, \emph{An introduction to {D}iophantine approximation},
  Cambridge Tracts in Math. and Math. Phys., no.~45, Cambridge University
  Press, 1957.

\bibitem[Kol53]{Kol53}
A.~N. Kolmogorov, \emph{On dynamical systems with an integral invariant on the
  torus}, Doklady Akad. Nauk SSSR (N.S.) \textbf{93} (1953), 763--766.

\bibitem[Kol54]{Kol54}
A.N. Kolmogorov, \emph{On the preservation of conditionally periodic motions
  for a small change in {H}amilton's function}, Dokl. Akad. Nauk. SSSR
  \textbf{98} (1954), 527--530.

\bibitem[Loc92]{Loc92}
P.~Lochak, \emph{{Canonical perturbation theory via simultaneous
  approximation}}, Russ. Math. Surv. \textbf{47} (1992), no.~6, 57--133.

\bibitem[Loc02]{Loc02}
\bysame, \emph{Simultaneous {D}iophantine approximation in classical
  pertubation theory: why and what for?}, Progress in nonlinear science, {V}ol.
  1 ({N}izhny {N}ovgorod, 2001), RAS, Inst. Appl. Phys., Nizhni\u\i\ Novgorod,
  2002, pp.~116--138.

\bibitem[Mos62]{Mos62}
J.~Moser, \emph{On {I}nvariant curves of {A}rea-{P}reserving {M}appings of an
  {A}nnulus}, Nachr. Akad. Wiss. Göttingen \textbf{II} (1962), 1--20.

\bibitem[Mos66]{Mos66}
\bysame, \emph{A rapidly convergent iteration method and non-linear
  differential equations. {II}}, Ann. Scuola Norm. Sup. Pisa (3) \textbf{20}
  (1966), 499--535.

\bibitem[Pös01]{Pos01}
J.~Pöschel, \emph{A lecture on the classical {KAM} theory}, Katok, Anatole
  (ed.) et al., Smooth ergodic theory and its applications (Seattle, WA, 1999).
  Providence, RI: Amer. Math. Soc. (AMS). Proc. Symp. Pure Math. 69, 707-732,
  2001.

\bibitem[Pös11]{Pos11}
\bysame, \emph{K{A}{M} à la {R}}, Regul. Chaotic Dyn. \textbf{16} (2011),
  no.~1-2, 17--23.

\bibitem[Rüs10]{Rus10}
H.~Rüssmann, \emph{{KAM}-iteration with nearly infinitely small steps in
  dynamical systems of polynomial character}, Discrete Contin. Dyn. Syst. Ser.
  S \textbf{3} (2010), no.~4, 683--718.

\bibitem[Sch80]{Sch}
W.~Schmidt, \emph{Diophantine approximation}, Lecture Notes in Math., no. 785,
  Springer, 1980.

\bibitem[Sev03]{Sev03}
M.~B. Sevryuk, \emph{The classical {KAM} theory at the dawn of the twenty-first
  century}, Mosc. Math. J. \textbf{3} (2003), no.~3, 1113--1144, 1201--1202,
  \{Dedicated to Vladimir Igorevich Arnold on the occasion of his 65th
  birthday\}.

\end{thebibliography}

\end{document}